\documentclass[12pt,a4paper,oneside]{amsart}

\usepackage{amsfonts, amsmath, amssymb, amsthm, amscd, hyperref}
\usepackage{anysize}
\usepackage{graphicx}

\newtheorem{theorem}{Theorem}[section]
\newtheorem{lemma}[theorem]{Lemma}

\newtheorem{proposition}[theorem]{Proposition}
\newtheorem{fact}[theorem]{Fact}

\theoremstyle{definition}
\newtheorem{definition}[theorem]{Definition}

\theoremstyle{remark}
\newtheorem{remark}[theorem]{Remark}
\newtheorem{question}[theorem]{Question}

\numberwithin{equation}{section}

\DeclareMathOperator{\volu}{vol}
\DeclareMathOperator{\diam}{diam}

\DeclareMathOperator{\inte}{int}
\DeclareMathOperator{\conv}{conv}
\DeclareMathOperator{\pos}{pos}

\begin{document}

\title{Equality cases in Viterbo's conjecture and isoperimetric billiard inequalities}

\author[A.~Balitskiy]{Alexey~Balitskiy}

\email{balitski@mit.edu}
\thanks{The author is supported by the Russian Foundation for Basic Research Grant 18-01-00036.}

\address{Dept. of Mathematics, Massachusetts Institute of Technology, 182 Memorial Dr., Cambridge, MA 02142}
\address{Dept. of Mathematics, Moscow Institute of Physics and Technology, Institutskiy per. 9, Dolgoprudny, Russia 141700}
\address{Institute for Information Transmission Problems RAS, Bolshoy Karetny per. 19, Moscow, Russia 127994}

%\subjclass[2010]{52A20, 52A23, 53D35}
\keywords{Billiards, Minkowski norm, Viterbo's conjecture, Permutohedron.}

\begin{abstract}
We apply the billiard technique to deduce some results on Viterbo's conjectured inequality between the volume of a convex body and its symplectic capacity. We show that the product of a permutohedron and a simplex (properly related to each other) delivers equality in Viterbo's conjecture. Using this result as well as previously known equality cases, we prove some special cases of Viterbo's conjecture and interpret them as isoperimetric-like inequalities for billiard trajectories.
\end{abstract}

\maketitle

\section{Introduction}

Claude Viterbo~\cite{viterbo2000metric} conjectured an isoperimetric inequality for any (normalized) symplectic capacity and the volume of a convex body $X\subset \mathbb R^{2n}$:
\begin{equation}
\label{eq:viterbo}
\volu (X) \ge \frac{c(X)^n}{n!}.
\end{equation}
The minimum is supposed to be attained (perhaps, not uniquely) on the symplectic images of the Euclidean ball. This inequality is proven up to a constant factor to the power of $n$ in~\cite{artstein2008m}. We investigate the case when $c$ is the Hofer--Zehnder symplectic capacity $c_{HZ}$.

The computation of symplectic capacities is considered difficult, so we restrict to the case of convex bodies that are Lagrangian products of a component in $V = \mathbb R^n$ with $q$-coordinates and a component in $V^* = \mathbb R^n$ with $p$-coordinates. For such bodies there is a simple geometric interpretation, established in~\cite{artstein2014bounds}, of the Hofer--Zehnder capacity. Namely,
%\begin{equation}
%\label{eq:ostrover}
$$
c_{HZ}(K \times T) = \xi_T(K),
$$
%\end{equation}
where $\xi_T(K)$ denotes the length of the shortest closed billiard trajectory in a convex body $K \subset V$ with geometry of lengths determined by another convex body $T \subset V^*$ by the formula
$$
\|q\|_T = \max_{p\in T} \langle p, q \rangle, \quad q \in V.
$$
Here we assume that $T$ contains the origin in the interior, but we do not assume that $T$ is symmetric, so our ``norm'' function $\|\cdot\|_T$ is not symmetric in general. The shortest closed billiard trajectory will be understood in the sense of K. Bezdek and D. Bezdek~\cite{bezdek2009shortest} as the closed polygonal line of minimal $\|\cdot\|_T$-length not fitting into a translate of the interior of $K$. In Section~\ref{sec:billiard} we introduce this billiard technique in detail. If we take $K$ to be symmetric with respect to the origin, and if we set $T = K^\circ = \{p \in V^*: \langle q, p \rangle \le 1 ~\forall q\in K\}$ to be the polar body of $K$, then inequality~(\ref{eq:viterbo}) becomes
\begin{equation}
\label{eq:mahler}
\volu (K \times K^\circ) \ge \frac{c_{HZ}(K\times K^\circ)^n}{n!},
\end{equation}
which is equivalent (as proven in~\cite{artstein2014from}) to the longstanding Mahler's conjecture in convex geometry. This is one of the reasons why inequality~(\ref{eq:viterbo}) might be considered difficult and interesting.

Our first result concerns certain convex bodies for which the equality in~(\ref{eq:viterbo}) holds and for which we do not know whether they are symplectic balls or not. One known (since~\cite{artstein2014from}) family of such examples consists of bodies $X = H \times H^\circ$, where $H$ is a Hanner (or Hanner--Hansen--Lima) polytope. The Hanner bodies are famous for being minimizers in Mahler's conjectured inequality: they deliver equality in~(\ref{eq:mahler}).
We introduce another family of examples. In Section~\ref{sec:equality}, we prove that
$$
\volu (P_n \times \triangle_n^\circ) = \frac{c_{HZ}(P_n \times \triangle_n^\circ)^n}{n!}.
$$
Here $\triangle_n$ denotes a regular $n$-dimensional simplex centered at the origin; $\triangle_n^\circ$ is its polar simplex; $P_n$ denotes a certain $n$-dimensional permutohedron, which can be thought of as the Minkowski sum of all the edges of $\triangle_n$. For example, permutohedron $P_2$ is a regular hexagon; permutohedron $P_3$ is a truncated regular octahedron whose hexagonal facets are regular hexagons (Figure~\ref{pic:permutohedron}). We discuss multiple definitions of permutohedron, and their relevant properties, in Section~\ref{sec:permutohedron}.

\begin{figure}[h]
\centering
\includegraphics[width=0.4\textwidth]{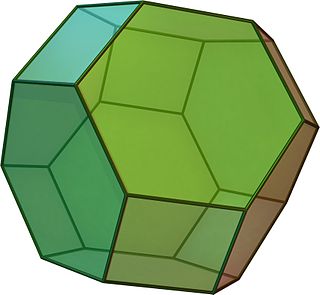}
\caption[]{Permutohedron $P_3$ is a properly truncated octahedron\footnotemark}
\label{pic:permutohedron}
\end{figure}

\footnotetext{The figure by w:en:User:Cyp@wikimedia, distributed under CC BY-SA 3.0 license.}

The rest of results in this paper are some very special cases of Viterbo's conjecture. In Section~\ref{sec:inequality}, we prove that Viterbo's conjecture (for $c_{HZ}$) holds for the bodies of the form $K \times \text{(a simplex $\triangle_n$)}$ or $K \times \text{(a parallelotope $\square_n$)}$, where $K \subset \mathbb{R}^n$ is any convex body, and the products are Lagrangian:
$$
\volu (K \times \triangle_n) \ge \frac{c_{HZ}(K \times \triangle_n)^n}{n!},
$$
$$
\volu (K \times \square_n) \ge \frac{c_{HZ}(K \times \square_n)^n}{n!}.
$$
The latter can be interpreted as a sharp isoperimetric-like inequality for billiard trajectories in the $\ell_1$-norm:
$$
\xi_{\square_n}(K) \le \left(2^n n! \volu (K)\right)^{1/n}.
$$
%Indeed, in view of~\ref{eq:ostrover}, $c_{HZ}(K \times \square_n) = \xi_{\square_n}(K)$ can be interpreted as
It seems that in other norms (corresponding to a fixed $T$, possibly non-symmetric) there are no known sharp inequalities bounding $\xi_T(K)$ in terms of $\volu K$. For example, in the Euclidean case ($T = B^n$ is the unit Euclidean $n$-ball) the inequality
$$
\xi_{B^n}(K) \le c_n \volu (K) ^{1/n}
$$
holds for any convex body $K \subset \mathbb{R}^n$, possibly non-symmetric. This inequality follows from the results of~\cite{artstein2008m} with factor $c_n = c \sqrt{n}$, for some absolute constant $c$.
But the optimal value of $c_n$, as well as the corresponding ``critical'' body $K$, seems to be unknown even for $n=2$.

\textbf{Acknowledgments.}
The author thanks Roman~Karasev for his constant attention to this work. The author also thanks Felix Schlenk for his helpful comments.

\section{Billiards in a Minkowski norm}
\label{sec:billiard}

Billiard dynamics in Minkowski (more generally, in Finsler) geometry corresponds physically to the propagation of light in a homogeneous anisotropic (respectively, in an inhomogeneous anisotropic) medium, with reflections in the boundary according to the Huygens principle (see~\cite{gutkin2002billiards}). We explain the relevant definitions in this section, modifying them for the non-smooth non-symmetric case.

We work in a pair of $n$-dimensional real vector space $V = \mathbb{R}^n$ and $V^* = \mathbb{R}^n$ with a canonical perfect pairing $\langle \cdot, \cdot \rangle$. If we identify $V$ with $V^*$ by sending a basis of $V$ to the dual basis of $V^*$, the pairing $\langle \cdot, \cdot \rangle$ becomes an inner product in $V$. The Euclidean norm $|v| = \sqrt{\langle v, v \rangle}$ will be denoted throughout the paper by the single bars. The double bars will denote a Minkowski norm (or a gauge function), as introduced below.

A \emph{convex body} $K \subset V = \mathbb{R}^n$ is a compact convex set with non-empty interior. The \emph{polar body} to a body $K \subset V$ is defined as $K^\circ = \{p \in V^*: \langle q, p \rangle \le 1 ~\forall q\in K\}$.

Let $V$ be endowed with the ``norm'' with unit ball $T^\circ = \{q \in V: \langle q, p \rangle \le 1 ~\forall p\in T\}$. We follow the notation of~\cite{akopyan2016elementary} and denote this norm by $\|\cdot\|_T$.
By definition, $\|q\|_T = \max\limits_{p \in T} \langle p, q \rangle$, where $\langle \cdot, \cdot \rangle : V^* \times V \rightarrow \mathbb{R}$ is the canonical bilinear form of the duality between $V$ and $V^*$. Here we assume that $T$ contains the origin but is not necessarily centrally symmetric. Therefore, our norms might not be symmetric; in general, $\|q\|_T \neq \|-q\|_T$. (Sometimes, such ``norms'' are called \emph{gauges}.)

The \emph{momentum} $p \in \partial T \subset V^*$ of the trajectory fragment $q \to q'$ is defined as a linear functional reaching its maximum at $q' - q$. If $T$ is not strictly convex, then there is an ambiguity in the definition of $p$.

The cone $N_K(q)$ of outer normals is defined as
$$
N_K(q) = \{n \in V^*: \langle n, q' - q\rangle \le 0 \ \forall q' \in K\}.
$$

The \emph{generalized reflection law} is the following relation:
\begin{equation}
\label{eq:reflection}
p' - p \in -N_K(q),
\end{equation}
where $p$ and $p'$ stand for the momenta of the billiard trajectory before and after the reflection at the point $q$.

\begin{definition}
\label{def:billiard}
A closed polygonal line $q_0 \to q_1 \to \ldots \to q_m = q_0$, with $q_i \in \partial K$, $q_i \neq q_{i+1}$, will be called a (generalized) \emph{closed billiard trajectory} in the configuration $K\times T$ if one can choose momenta $p_i \in \partial T$ for $q_{i} \to q_{i+1}$, $0 \le i < m$, so that the generalized reflection law (\ref{eq:reflection}) holds for each bounce $q_i$.
\end{definition}

By a \emph{classical billiard trajectory} we mean a trajectory that meets only smooth points of $\partial K$ and smooth points of $\partial T$. Additionally, in what follows we do not allow a classical trajectory to pass the same path multiple times.

We set
$$
\mathcal P_m(K) = \{ (q_1,\ldots, q_m) : \{q_1,\ldots, q_m\} \ \text{doesn't fit into}\ (\inte K + t)\ \text{with} \ t\in V \} =
$$
$$
= \{(q_1,\ldots, q_m) : \{q_1,\ldots, q_m\} \ \text{doesn't fit into}\ (\alpha K+t)\ \text{with}\ \alpha\in (0,1),\ t\in V \}
$$
and
$$
\xi_T(K) = \min_{Q \in \mathcal Q_T(K)} \ell_T(Q),
$$
where $Q = (q_1,\ldots, q_m),\ m \ge 2,$ ranges over the set $\mathcal Q_T(K)$ of all closed generalized billiard trajectories in $K$ with geometry defined by $T$. (Here we denote the length $\ell_T (q_1,\ldots, q_m) = \sum_{i=1}^m \|q_{i+1} - q_i\|_T$.)

It might not be clear if the minimum in the definition of $\xi_T(K)$ is attained, but it is true, and we can say even more:

\begin{theorem}
\label{thm:bezdeks}
For any convex bodies $K \subset V$, $T \subset V^*$ ($T$ is smooth) containing the origins of $V$ and $V^*$ in their interiors, the following holds:
$$
\xi_T(K) = \min_{m\ge 2} \min_{Q\in \mathcal P_m(K)} \ell_T(Q) = \min_{2\le m\le n+1} \min_{Q\in \mathcal P_m(K)} \ell_T(Q).
$$
%Moreover, any minimizer of the right-hand side is a shortest closed billiard trajectory of length $\xi_T(K)$ (possibly, after a translation).
%and furthermore, the minimum over $m$ is attained at $m\le n + 1$.
\end{theorem}

\begin{remark}
\label{rem:xidef}
Theorem~\ref{thm:bezdeks} was proved in~\cite{bezdek2009shortest} in the Euclidean case (when $T$ is the Euclidean ball). In~\cite{akopyan2016elementary} it was generalized for the case of smooth bodies $K$, $T$.
We obtain the formulation above by approximating non-smooth bodies by smooth ones in the Hausdorff metric and passing to the limit.
%Note that the formula of Theorem~\ref{thm:bezdeks} can be used as the definition of $\xi_T(K)$ for arbitrary $T$ and $K$ without any smoothness assumptions.
\end{remark}

\section{Permutohedron properties}
\label{sec:permutohedron}

In this section, we recollect briefly several equivalent definitions of a (regular) permutohedron. We also explain some of its properties that we will need later.

The classical definition is the following one:

\begin{definition}
\label{def:permut1}
The $n$-dimensional permutohedron is the convex hull of points
$$
(\sigma(1), \sigma(2), \ldots, \sigma(n+1)) \in \mathbb{R}^{n+1},
$$
over all permutations $\sigma : \{1,2,\ldots,n+1\} \to \{1,2,\ldots,n+1\}$.
\end{definition}

For the proof of Theorem~\ref{thm:viterboeq}, the following definition will be a convenient one (it can be found in~\cite[Lecture~7.3]{ziegler1995lectures}):

\begin{definition}
\label{def:permut2}
Consider the regular simplex $\triangle_n = \conv\{v_0, \ldots, v_n\} \subset \mathbb{R}^n$ centered at the origin and normalized so that its edges are all of unit (Euclidean) length. The permutohedron $P_n \subset \mathbb{R}^n$ is defined as the Minkowski sum of the simplex edges:
$$
P_n = \sum_{0\le i<j \le n} [v_i,v_j].
$$
\end{definition}

For the proof of Theorem~\ref{thm:viterbosimplex}, the following definition will be of use (it can be found in~\cite[Chapter~21]{conway2013sphere}):

\begin{definition}
\label{def:permut3}
The permutohedron $\widetilde{P_n}$ is defined as the Vorono\u{\i} cell of the lattice
$$
A_n^* = \left\{(x_0, \ldots, x_n) \in \mathbb{Z}^{n+1}: \ \sum\limits_{i} x_i = 0, \ x_0 \equiv \ldots \equiv x_n \pmod{n+1}\right\},
$$
lying in $\{(x_0, \ldots, x_n) \in \mathbb{R}^{n+1}: \ \sum\limits_{i} x_i = 0\} \cong \mathbb{R}^n$.
\end{definition}

\begin{remark}
\label{rem:lattice} Observe that the lattice $A_n^*$ is generated by the vectors
$$
a_1 = (\underbrace{-1, n, -1, \ldots, -1}_{n+1})^t, \ldots, a_{n} = (-1,-1,-1,\ldots, n)^t.
$$
This lattice also contains the vector
$$
a_0 = (\underbrace{n, -1, -1, \ldots, -1}_{n+1})^t = -a_1 - \ldots -a_n.
$$
Those vectors $a_0, a_1, \ldots, a_n$ are the only vectors of the shortest nonzero length in $A_n^*$, as one can check manually.
\end{remark}

All these definitions give the same result up to similarity. The width of $P_n$ equals $n \langle v_0-v_1, e_n\rangle = \sqrt{(n^2+n)/2}$, and the width of $\widetilde{P_n}$ equals $|a_0| = \sqrt{n^2+n}$ (the fact that those lengths are indeed widths will be justified below, in part (3) of Fact~\ref{fact:combifacet}). Comparing the widths, we see that $\widetilde{P_n}$ is $\sqrt{2}$ times larger than $P_n$.

\begin{fact}
\label{fact:volume}
In the notation of Definition~\ref{def:permut2},
$$
\volu P_n = \frac{(n+1)^{n-1/2}}{2^{n/2}}.
$$
\end{fact}

\begin{proof}
%The volume $\volu \triangle_n = \frac{\sqrt{n+1}}{2^{n/2} n!}$ can be easily computed, and the Mahler volume product $\volu \triangle_n \cdot \volu \triangle_n^\circ = \frac{(n+1)^{n+1}}{(n!)^2}$ is well known, so we do not give details here.

We have:
$$
 2^{n/2} \volu P_n = \volu \widetilde{P_n} = \det A_n^* = \det \Gamma(a_1, \ldots, a_n)^{1/2} =
 \begin{vmatrix}
  n^2 + n & -n-1 & -n-1 & \vdots \\
 -n-1 & n^2+n & -n-1 & \vdots\\
 -n-1 & -n-1 & n^2+n & \vdots \\
 \cdots & \cdots & \cdots & \ddots
 \end{vmatrix}^{1/2}.
$$
The latter $n\times n$ Gram matrix has an eigenvector $h = (1, \ldots, 1)^t$ corresponding to the eigenvalue $n+1$ without multiplicities. All the other eigenvectors, orthogonal to $h$, have the same $(n-1)$-fold eigenvalue $(n+1)^2$, as it is easy to check by hand. Therefore, $\det \Gamma(a_1, \ldots, a_n) = (n+1)^{2n-1}$, and the result follows.
\end{proof}

Now let us discuss the well-known combinatorial structure of permutohedron. The following fact is essentially a reformulation of~\cite[Proposition~2.6]{postnikov2009permutohedra}, where it is formulated in terms of Definition~\ref{def:permut1} and we translate it into the language of Definition~\ref{def:permut3}.

\begin{fact}
\label{fact:combipermut}
Let $\widetilde{P_n}$ be the Vorono\u{\i} cell of $A_n^*$ around the origin, as in Definition~\ref{def:permut3}. Fix numbers $y_i = \frac{n}{2} - i$, $0 \le i \le n$.
\begin{enumerate}
\item
The $d$-dimensional faces of $\widetilde{P_n}$ are in one-to-one correspondence with the ordered partitions of $\{0,1,\ldots,n\}$ into $n+1-d$ disjoint subsets. Given a decomposition $\{0,1,\ldots,n\} = B_0 \cup B_1 \cup \ldots \cup B_{n-d}$, the corresponding $d$-face of $\widetilde{P_n}$ consists of all points $(x_0, \ldots, x_n) \in \mathbb{R}^{n+1}$ satisfying for all $I \subseteq \{0,1,\ldots,n\}$
$$
\sum\limits_{i \in I} x_i \le \sum\limits_{j = 0}^{|I|-1} y_j,
$$
and such that the equality is attained (at least) for $I = B_0$, $I = B_0 \cup B_1$, $\ldots$, $I = B_0 \cup B_1 \cup \ldots \cup B_{n-d}$.

\item In particular, the vertices of $\widetilde{P_n}$ are in one-to-one correspondence with the ordered partitions into singletons. Having such a partition $B_0 \cup B_1 \cup \ldots \cup B_{n}$, associate to it the permutation $\pi$ of $\{0,1,\ldots,n\}$ given by $\pi(i) = j$ if $i \in B_j$. Given a permutation $\pi : \{0,1,\ldots,n\} \to \{0,1,\ldots,n\}$, let us write $v_\pi$ for the corresponding vertex. Then $v_\pi$ has coordinates $(y_{\pi(0)}, y_{\pi(1)}, \ldots, y_{\pi(n)})$. %Additionally, two vertices $v_\pi, v_{\pi'}$ are connected by an edge if and only if $\pi' = \pi s_i$, for some adjacent transposition $s_i = (i,i+1)$, $0\le i < n$.

\item The face corresponding to a partition $B$ contains the face corresponding to a partition $B'$ if and only if $B'$ is a refinement of $B$.

\end{enumerate}
\end{fact}

We give more details regarding the facets of $\widetilde{P_n}$.

\begin{fact}
\label{fact:combifacet}
Let $\widetilde{P_n}$ be the Vorono\u{\i} cell of $A_n^*$ around the origin, as in Definition~\ref{def:permut3}.
\begin{enumerate}
\item The facets of $\widetilde{P_n}$ are in one-to-one correspondence with the ordered pairs of sets $(S, \{0,1,\ldots,n\}\setminus S)$, for $\varnothing \neq S \subsetneq \{0,1,\ldots,n\}$. Let us write $F_S$ for the facet corresponding to $(S, \{0,1,\ldots,n\}\setminus S)$.

\item For any $S$, an outward pointing normal for $F_S$ can be chosen from the lattice $A_n^*$. Explicitly, we can take
 $$
 a_S = \sum_{i \in S} a_i.
 $$
 as an outer normal to $F_S$. The distance from the origin to $F_S$ equals $\frac{1}{2} |a_S|$.
 %If $S = \{i\}$, the corresponding facet's outer normal is $a_i$. If $S = \{0,1,\ldots,n\} \setminus \{i\}$, the corresponding facet's outer normal is $-a_i$.

\item The closest to the origin facets are characterized as follows: they are congruent to an $(n-1)$-dimensional permutohedron, and the corresponding set $S$ is of cardinality $|S|=1$ or $|S|=n$. The width of $\widetilde{P_n}$ is $|a_0|$.

\item Let $v_\pi$ be a vertex of $\widetilde{P_n}$ lying in a facet $F_S$ (so that $\pi(S) = \{0, \ldots, |S|-1\}$). Then $v_\pi-a_S$ is another vertex of $\widetilde{P_n}$, and its corresponding permutation $\pi'$ is obtained from $\pi$ by the left cyclic $|S|$-shift of values (so that $\pi(S) = \{n - |S| + 1, \ldots, n\}$). That is, $\pi'(i) = \pi(i) - |S| \pmod{n+1}$.

%\item Let $S = \{i\}$, and let $v_\pi$ be a vertex of $\widetilde{P_n}$ lying in $F_S$ (so that $\pi(i) = 0$). Then $v_\pi-a_i$ is another vertex of $\widetilde{P_n}$, and its corresponding permutation $\pi'$ is obtained from $\pi$ by the left cyclic shift of values (so that $\pi'(i) = n$). Similarly, if $S = \{0,1,\ldots,n\} \setminus \{i\}$, and $v_\pi \in F_S$, then $v_\pi+a_i = v_{\pi'}$, where $\pi'$ is obtained from $\pi$ by the right cyclic shift of values.

\end{enumerate}
\end{fact}

\begin{proof}
\begin{enumerate}
\item Follows from Fact~\ref{fact:combipermut}, part (1).
\item It follows from the definition of the Vorono\u{\i} tessellation that for any facet $F_S$ of $\widetilde{P_n}$, there is an adjacent Vorono\u{\i} cell touching $\widetilde{P_n}$ by $F_S$. The vector from the origin to the center of this cell is orthogonal to $F_S$ and is cut by $F_S$ into equal halves. This vector belongs to the lattice and can be taken as an outer normal for $F_S$. We want to show that it coincides with $a_S = \sum\limits_{i \in S} a_i$. On one hand, immediately from Fact~\ref{fact:combipermut} we get that all the vectors $a_i-a_j$ where either $i,j\in S$ or $i,j \notin S$, are parallel to $F_S$. Among those $a_i-a_j$, there are $(n-1)$ linearly independent vectors, so they span the tangent hyperplane for $F_S$. Also they are all orthogonal to $\sum_{i \in S} a_i$ (by a direct computation, using $\langle a_i, a_j\rangle = - n - 1 + (n+1)^2\delta_{ij}$). On the other hand, the vector $a_S = \sum\limits_{i \in S} a_i$ is primitive in the sense that $a_S \neq r\lambda$ for $r > 1$, $\lambda \in A_n^*$. The result now follows.
\item The distance from the origin to a facet $F_S$ equals $\frac12 |a_S|$. So the statement follows now from Remark~\ref{rem:lattice}: the shortest nonzero vectors of $A_n^*$ are precisely $a_0, a_1, \ldots, a_n$.
\item Let $v_\pi \in F_S$. The coordinates of $v_\pi$ are $(y_{\pi(0)}, y_{\pi(1)}, \ldots, y_{\pi(n)})$. Then the $i$-th coordinate ($0 \le i \le n$) of $v_\pi - a_S$ is
$$
(v_\pi - a_S)_i =
\begin{cases}
y_{\pi(i)} - n + |S|-1, & i \in S, \\
y_{\pi(i)} + |S|, & i \notin S.
\end{cases}
$$
These are precisely the coordinates of $v_{\pi'}$, where $\pi'(i) = \pi(i) - |S| \pmod{n+1}$.

%\item Let $S = \{i\}$ and $v_\pi \in F_S$. Then $\pi(i) = 0$ and the coordinates of $v_\pi$ are given by
%$$
%\left(\frac{n}{2} - \pi(0), \ldots, \underbrace{\frac{n}{2} - \pi(i)}_{= n/2}, \ldots, \frac{n}{2} - \pi(n)\right).
%$$
%It can be seen that $v_\pi-a_i$ has the coordinates $(y_{\pi'(0)}, y_{\pi'(1)}, \ldots, y_{\pi'(n)})$, where $\pi'(i) = n$, $\pi'(j) = \pi(j) -1$ for $j \neq i$.
%The case $S = \{0,1,\ldots,n\} \setminus \{i\}$ is similar.
\end{enumerate}
\end{proof}

Now we use this combinatorial description in order to establish a lemma that we will need in Section~\ref{sec:inequality}.

Consider the vectors $u_0, \ldots, u_n$ such that $|u_0| = \ldots = |u_n| = 1$ and the directions of the $u_i$ are equiangular, that is, the endpoints of $u_0, \ldots, u_n$ form a regular simplex in $\mathbb{R}^n$. Consider the lattice $\Lambda$ generated by the vectors $u_1, \ldots, u_n$. Since $u_0 = -\sum\limits_{i=1}^n u_i$, we have $u_0 \in \Lambda$. Let $P$ be the Vorono\u{\i} cell of $\Lambda$ around the origin. Observe that $\Lambda$ is just a scaled copy of $A_n^*$, so comparing $|u_0| = 1$ with $|a_0| = \sqrt{n^2+n}$ we obtain that $P$ is congruent to $\frac{1}{\sqrt{n^2+n}} \widetilde{P_n}$. We keep all the notation (like $v_\pi$, $F_S$) of Facts~\ref{fact:combipermut},~\ref{fact:combifacet} in the context of $P$.

Consider the tiling $T$ of $\mathbb{R}^n$ that is dual to the Vorono\u{\i} tessellation of $\mathbb{R}^n$ with respect to $\Lambda$. It is also known as the Delaunay tiling (see, e.g.,~\cite[Chapter~32]{gruber2007convex}). Let us recall its construction. The vertices of $T$ are precisely the elements of $\Lambda$. Vertices $\lambda_0, \lambda_1, \ldots, \lambda_d \in \Lambda$, $0 \le d \le n$, form a $d$-face of $T$ if the Vorono\u{\i} cells centered at the $\lambda_i$ have a common $(n-d)$-face.
It turns out that $T$ is in fact a triangulation (i.e., the corresponding polyhedral complex is simplicial). This might not be clear a priori, so we prove it as a part of the following lemma.

%To see that, we exploit the combinatorial description of the permutohedral Vorono\u{\i} tessellation, discussed above. Every full-dimensional tile of $T$ corresponds to a vertex $v$ of the permutohedral tesselation. We would like to show that $v$ is adjacent to $n+1$ permutohedra.

%Let $P$ be a Vorono\u{\i} cell of $\Lambda$ centered in $\lambda \in \Lambda$. We know that $P$ is congruent to $p_n$.

\begin{lemma}
\label{lem:equiedge}
Every full-dimensional cell of $T$ is a simplex. Every simplex $\sigma$ of $T$ has the following property: There exist a closed oriented polygonal line $Q_\sigma$ such that
\begin{itemize}
\item $Q_\sigma$ traverses along edges of $\sigma$ and visits every vertex of $\sigma$ once; in particular, $\conv Q_\sigma = \sigma$;
\item the segments of $Q_\sigma$ have unit length and the set of their directions coincides with $\{v_0, \ldots, v_n\}$.
\end{itemize}

\end{lemma}

\begin{proof}
Let $\sigma$ be a full-dimensional cell of $T$ corresponding to a vertex $v$ of the permutohedral tessellation. First, we want to prove that $v$ is adjacent to $n+1$ permutohedra. Let $P'$ and $P''$ be two of them sharing a common facet $F$, and let $\lambda', \lambda'' \in \Lambda$ be their centers. Note that $v - \lambda'$ is a vertex of $P$, the Vorono\u{\i} cell around the origin. We indexed the vertices of $P$ by permutations in Fact~\ref{fact:combipermut}. Let $\pi'$ be the relevant permutation. Similarly we introduce $\pi''$. How are $\pi'$ and $\pi''$ related? Firstly, they are different, because $v_{\pi'} = v - \lambda' \neq v - \lambda'' = v_{\pi''}$. Secondly, if $F - \lambda' = F_{S'}$, then $\lambda''-\lambda'$ is the outer normal for the face $F_{S'}$ of $P$ constructed in the proof of Fact~\ref{fact:combifacet}, part (2). We know that $v_{\pi'} - (\lambda''-\lambda') = v_{\pi''}$, and using Fact~\ref{fact:combifacet}, part (4), we conclude that $\pi''$ is obtained from $\pi'$ by a multiple left cyclic shift of values. Therefore, all the permutohedra adjacent to $v$ correspond to different cyclic shifts of a single permutation of $\{0,1,\ldots,n\}$. Hence, there are $n+1$ of them, and $\sigma$ is a simplex.

Now we construct $Q_\sigma$, for a simplex $\sigma$ of the Delaunay triangulation $T$. Let $v$ be the vertex of the permutohedral tessellation corresponding to $\sigma$. Let $\lambda \in \Lambda$ be a vertex of $\sigma$. We would like to show that among the edges of $\sigma$ incident to $\lambda$, there are precisely two of length 1, and one of them is pointed along a vector $u_i$ while the other is pointed against a vector $u_j$, for some $i\neq j$. Let $\pi$ be the permutation of $\{0,1,\ldots,n\}$ corresponding to the vertex $v - \lambda$ of $P$. Recall from Fact~\ref{fact:combifacet}, part (3), that there are exactly two facets of $P$ containing $v_\pi$ and congruent to a $(n-1)$-dimensional permutohedron. They correspond to the pairs $(\{\pi(0)\}, \{\pi(1), \ldots, \pi(n)\})$ and $(\{\pi(0), \ldots, \pi(n-1)\}, \{\pi(n)\})$. These two facets correspond to the shortest edges of $\sigma$ pointed along, say, $v_i$ and $-v_j$. From Fact~\ref{fact:combifacet}, part (4), we get that the vertices $v_\pi - u_i$ and $v_\pi + u_j$ of $P$ correspond to the left and right cyclic shifts of $\pi$. Let us mark all those (oriented along the $u_0, \ldots, u_n$) edges of $\sigma$ that we found, over all vertices $\lambda$ of $\sigma$. They form a family of oriented cycles passing through each vertex $\lambda \in \sigma$ once. We claim that this family is in fact a single cycle. To see that, we keep track of the permutations $\pi$ corresponding to the current vertex of $\sigma$. As we move along a cycle from the family, these permutations shift cyclically, as described above. Thus, the cycle consists of $(n+1)$ segments, and we found $Q_\sigma$ as desired. The lemma is proven.
\end{proof}

\section{Equality cases in Viterbo's conjecture}
\label{sec:equality}

Consider the regular simplex $\triangle_n = \conv\{v_0, \ldots, v_n\} \subset \mathbb{R}^n$ centered at the origin and normalized so that its edges are all of unit (Euclidean) length. Choose an orthonormal base $(e_1, \ldots, e_n)$ with $e_n$ pointing to $v_0$. Consider also the permutohedron $P_n \subset \mathbb{R}^n$, as in Definition~\ref{def:permut2}:
$$
P_n = \sum_{0\le i<j \le n} [v_i,v_j].
$$

The main result of this section if the following

\begin{theorem}
\label{thm:viterboeq} In the configuration $P_n \times \triangle_n^\circ$ the shortest generalized billiard trajectory has length $\xi_{\triangle_n^\circ}(P_n) = (n+1)^2$. Moreover, $X = P_n \times \triangle_n^\circ$ delivers equality in Viterbo's conjecture (where $P_n$ and $\triangle_n^\circ$ lie in Lagrangian subspaces).
\end{theorem}

This theorem can be viewed as the extension of the first part of the following proposition, proved in~\cite{balitskiy2016shortest}:

\begin{proposition}
\label{prop:triangle}
\begin{enumerate}
\item
In the configuration $P_2 \times \triangle_2^{\circ}$ the shortest generalized billiard trajectory has length
$$
\xi_{\triangle_2^{\circ}}(P_2) = 9.
$$
Hence, $X = P_2 \times \triangle_2^{\circ}$ delivers equality in Viterbo's conjecture.

\item
Any classical billiard trajectory in the configuration
$$
P_2 \times \triangle_2^{\circ}
$$
bounces $4$ times and has length 9.

\item
Arbitrarily close to any point
$$
(q,p) \in \partial(P_2 \times \triangle_2^{\circ})
$$
there passes a certain classical billiard trajectory of minimal length.
\end{enumerate}
\end{proposition}

Another statement of the same spirit, proved in~\cite{balitskiy2016shortest} is the following

\begin{proposition}
\label{prop:hanner}
\begin{enumerate}
\item
In a Hanner polytope $H \subset \mathbb{R}^n$ with geometry specified by its polar $H^{\circ}$, the shortest generalized billiard trajectory has length $\xi_{H^\circ}(H) = 4$, and $X = H\times H^\circ$ delivers equality in Viterbo's conjecture (the latter fact has been known since~\cite{artstein2014from}).

\item
Any classical billiard trajectory in a Hanner polytope $H \subset \mathbb{R}^n$ with geometry specified by its polar $H^{\circ}$ is $2n$-bouncing and has length 4.

\item
Moreover, in an arbitrarily small neighborhood of any point $(q,p) \in \partial (H \times H^{\circ})$, there passes a classical billiard (in the configuration $H \times H^{\circ}$) trajectory of minimal length.
\end{enumerate}
\end{proposition}

F. Schlenk has established that the interior of the Lagrangian product of a crosspolytope and a cube (its polar) is ``almost'' a symplectic ball in the following sense: the product symplectically embeds into the ball of arbitrarily close volume, and vice versa. The four-dimensional case can be found in~\cite[Lemma~3.1.8]{schlenk2005embedding} or~\cite[\S 4]{latschev2013gromov}; the $2n$-dimensional case follows the same lines (see also~\cite[Lemma~5.3.1]{schlenk2005embedding}, where the same technique is exploited for a simplex instead of a crosspolytope). A reasonable question is if the Lagrangian product of a crosspolytope and a cube (or more generally, $H\times H^\circ$ for an arbitrary Hanner polytope) is a symplectic ball. However, it is not clear for me that the affirmative answer to this question would reprove the implications of Proposition~\ref{prop:hanner}. It would have reproved them if the boundary of the product had been smooth (in this case we would have applied the results of \cite{paiva2014contact}).

Let us also note that Theorem~\ref{thm:viterboeq} would immediately follow from the affirmative answer to the following

\begin{question}
\label{ques:permutball}
Is $P_n \times \triangle_n^{\circ}$ a symplectic ball? More precisely, is the interior of $P_n \times \triangle_n^{\circ}$ a symplectic image of a Euclidean ball of the same volume?
\end{question}

We do not know the answer already in the case $n=2$.
%The negative answer to Question~\ref{ques:permutball} would disprove another well-known conjecture. It states that all symplectic capacities coincide on the class of convex bodies. This conjecture is even stronger than~\ref{eq:viterbo}.

To prove Theorem~\ref{thm:viterboeq}, we will directly compute $\xi_{\triangle_n^\circ}(P_n) = c_{HZ}(P_n \times \triangle_n^\circ)$ and then check the equality:

$$
\volu (P_n \times \triangle_n^\circ) = \frac{c_{HZ}(P_n \times \triangle_n^\circ)^n}{n!}.
$$

The volume $\volu \triangle_n = \frac{\sqrt{n+1}}{2^{n/2} n!}$ can be easily computed, the Mahler volume product $\volu \triangle_n \cdot \volu \triangle_n^\circ = \frac{(n+1)^{n+1}}{(n!)^2}$ is well known, so we conclude that $\volu \triangle_n^\circ = \frac{2^{n/2} (n+1)^{n+1/2}}{n!}$. The volume $\volu P_n = \frac{(n+1)^{n-1/2}}{2^{n/2}}$ was already computed (Fact~\ref{fact:volume}).
So it suffices to compute the right-hand side.

\begin{proposition}
\label{prop:capacity}
In the above notation,
$$
c_{HZ} (P_n \times \triangle_n^\circ) = (n+1)^2.
$$
\end{proposition}

For the proof, we use the characterization of Bezdek and Bezdek of shortest generalized billiard trajectories in $P_n$, when lengths are measured using the norm with unit body $\triangle_n$.

Consider first the following 2-periodic trajectory: take the centers $m_1, m_2$ of two opposite facets of $P_n$ that are congruent to $P_{n-1}$. Then the generalized billiard trajectory $m_1 \to m_2 \to m_1$ is a 2-fold bypass of the width of $P_n$. Clearly, it cannot fit into $\inte P_n$. The $\triangle_n$-length of this trajectory is $(n+1)^2$, so we have an estimate from above: $c_{HZ} (P_n \times \triangle_n^\circ) \le (n+1)^2$.

To prove the estimate from below, we consider an arbitrary closed polygonal line that cannot fit into $\inte P_n$ and show that its $\triangle_n$-length cannot be less than $(n+1)^2$.

First, we replace each segment $[q,q']$ of this line with a certain polygonal line of same $\triangle_n$-length but with edges directed along $v_0, \ldots, v_n$. This can be done as follows. Consider the convex cone
$$
C_j = \pos\{\{v_0, \ldots, v_n\} \setminus \{v_j\}\}
$$
in which the vector $q'-q$ lies. Such $C_j$ exists since $\bigcup\limits_{i=0}^n C_i = \mathbb{R}^{2n}$.
Then, decompose $q'-q = \sum\limits_{0\le k \le n, \ k\neq j} \alpha_k v_k$ with $\alpha_k \ge 0$. Now the polygonal line with edges congruent to $\alpha_i v_i$ suits our purpose well. Its $\triangle_n$-length equals the $\triangle_n$-length of $q'-q$ in the given norm because the norm function is linear on $C_j$.

Now, we have the closed polygonal line with at most $n+1$ directions used. Let its segments be congruent to $\alpha_1 v_{j_1}, \ldots, \alpha_m v_{j_m}$. The total $\triangle_n$-length of the segments of this polygonal line along the direction $v_i$ is proportional to $\sum\limits_{j_k = i} \alpha_k$. Note that $\sum\limits_i v_i = 0$ is the only linear dependence among the directions, so the relation $\sum\limits_{k=1}^m \alpha_k v_{j_k} = 0$ is a multiple of $\sum\limits_i v_i = 0$. Thus, the total $\triangle_n$-length of the segments of this polygonal line along each direction $v_i$ does not depend on $i$. Assume the contrary to the statement that we are to prove: suppose that this length along each direction $v_i$ is less than $n+1$. We need to show that such a line fits into a smaller homothet of $P_n$, this will be a contradiction. We formulate it as a lemma.

\begin{lemma}
\label{lem:fitinto}
Suppose a closed polygonal line consists of segments directed only along $v_0, \ldots, v_n$. Suppose also that the total $\triangle_n$-length of all segments that are directed along $v_i$ equals $n+1$, for each $i = 0, 1, \ldots, n$. Then this line can be covered by a translate of $P_n$.
\end{lemma}

\begin{proof}
We proceed by induction on $n$. The base case $n = 1$ is clear.

To begin, we note that the top ``horizontal'' (meaning orthogonal to $e_n$, which is supposed to be pointed ``upwards'') facet $F$ of $P_n$ is congruent to a copy of $P_{n-1}$ placed horizontally in $\mathbb{R}^n$. Explicitly,
$$
F = n v_0 + \sum\limits_{1\le i < j\le n} [v_i, v_j].
$$

Now let $Q = q_0 \to q_1 \to \ldots \to q_{m-1} \to q_m = q_0$ be a polygonal line satisfying the assumptions of the lemma, and let $q_0$ be the highest (having the largest coordinate along $e_n$) vertex of $Q$. Let us introduce a parametrization $q : [0,n+1] \to Q$, so that the point $q(t)$ runs along $Q$ with the constant velocity $(n+1)$. Without loss of generality, $q(0) = q(n+1) = q_0$.
%, and the portion of time when the point travels along any $v_i$ direction equals 1.

Consider the following transformation of $Q$. We contract it using the transform
$$
\begin{pmatrix}
\frac{1}{n+1} & 0 & \vdots & 0 & 0\\
 0 & \frac{1}{n+1} & \vdots & 0 & 0\\
 \cdots & \cdots & \ddots & \cdots & \cdots\\
 0 & 0 & \vdots & \frac{1}{n+1} & 0\\
 0 & 0 & \vdots & 0 & 1\\
\end{pmatrix}
$$
and obtain the line $\widetilde Q$ with the corresponding parametrization $\widetilde q(\cdot)$. Also, denote by $\widetilde v_i$ the image of $v_i$ under this transform. Note that $(n+1)\widetilde v_i =  v_i-v_0$.

Now we consider the Minkowski sum
$$
R = F + \widetilde Q = \bigcup\limits_{t\in[0,n+1]} (F + \widetilde q(t))
$$
and claim that $Q$ can be covered by $R$ (see Figure~\ref{pic:permutfit1}). We will find $s \in \mathbb{R}^n$ such that $q(t) \in F+s+\widetilde q(t)$ for all $t$. We determine the vertical component $\langle s, e_n \rangle = - \frac{n+1}{2} |v_0|$ from the relation $q(t) \in F+s+\widetilde q(t)$. So we need to adjust the horizontal component of $s$ in order to satisfy $q(t) \in F+s+\widetilde q(t)$.

%Informally, the argument goes as follows. Imagine that we track the relative motion of the point $q(t)$ inside the moving horizontal affine hyperplane $H(t)$ spanned by permutohedron $F+s+\widetilde q(t)$. The relative velocity of $q(t)$ in $H(t)$ equals $\frac{n}{n+1}$-fraction of the corresponding horizontal velocity component of $q(t)$. Therefore, for this relative motion, the inductive assumption holds: In permutohedron $P_{n-1}$ we have the trajectory that travels by distance $n$ along each of the $n$ distinguished directions (defined as horizontal projections of $v_1, \ldots, v_n$). Therefore, such $s$ can indeed be found.
%
%Let us give a formal argument now.
Let $H = \mathbb{R}^{n-1} \times \{0\}$ be the horizontal subspace of $\mathbb{R}^n$. Let $\pi: \mathbb{R}^n \to H$ the orthogonal projection. Endow $H$ with the norm given by the unit body $\triangle_{n-1} = \conv\{\pi(v_1), \ldots, \pi(v_n)\}$. Identify $P_{n-1}$ with $\pi\left(F\right) = \sum\limits_{1\le i < j\le n} [\pi(v_i), \pi(v_j)] \subset H$. Consider the polygonal line $q(t) - \widetilde q(t)$, for $t\in[0,n+1]$. When $q(t)$ traverses along $v_i$, $1\le i \le n$, $q(t) - \widetilde q(t)$ traverses along $\pi(v_i)$. Whenever $q(t)$ traverses along $v_0$, $q(t) - \widetilde q(t)$ stays fixed. The overall time during which $q(t)$ traverses along $v_0$ is 1 (out of $n+1$). Let us reparametrize the polygonal line $q(t) - \widetilde q(t)$ by skipping all the time intervals during which it stays fixed. This way we will get a polygonal line $q'(t)$, for $t \in [0,n]$.
We make a few observations:
\begin{itemize}
\item $q'(t) \in H$ for all $t\in[0,n]$;
\item if we measure distances in $H$ using the $\triangle_{n-1}$-norm, then $q'(t)$ travels with the constant velocity $n$;
\item $q'(t)$ travels only along the distinguished directions $\pi(v_1), \ldots, \pi(v_n)$.
\end{itemize}
Therefore, we can apply the inductive assumption: there exists $s' \in H$ such that $q'(t) \in P_{n-1} + s'$ for all $t \in [0,n]$. Consequently, $q(t) - \widetilde q(t) \in P_{n-1} + s'$ for all $t \in [0,n+1]$.

Finally, we set $s = s' - \frac{n+1}{2} v_0 \in \mathbb{R}^n$. Then we have
$$
q(t) \in P_{n-1} + s' + \widetilde q(t) = F - \frac{n+1}{2} v_0 + s' + \widetilde q(t) = F+s+\widetilde q(t),
$$
so $Q \subset R + s$.

\begin{figure}[h]
\centering
\includegraphics[width=0.75\textwidth]{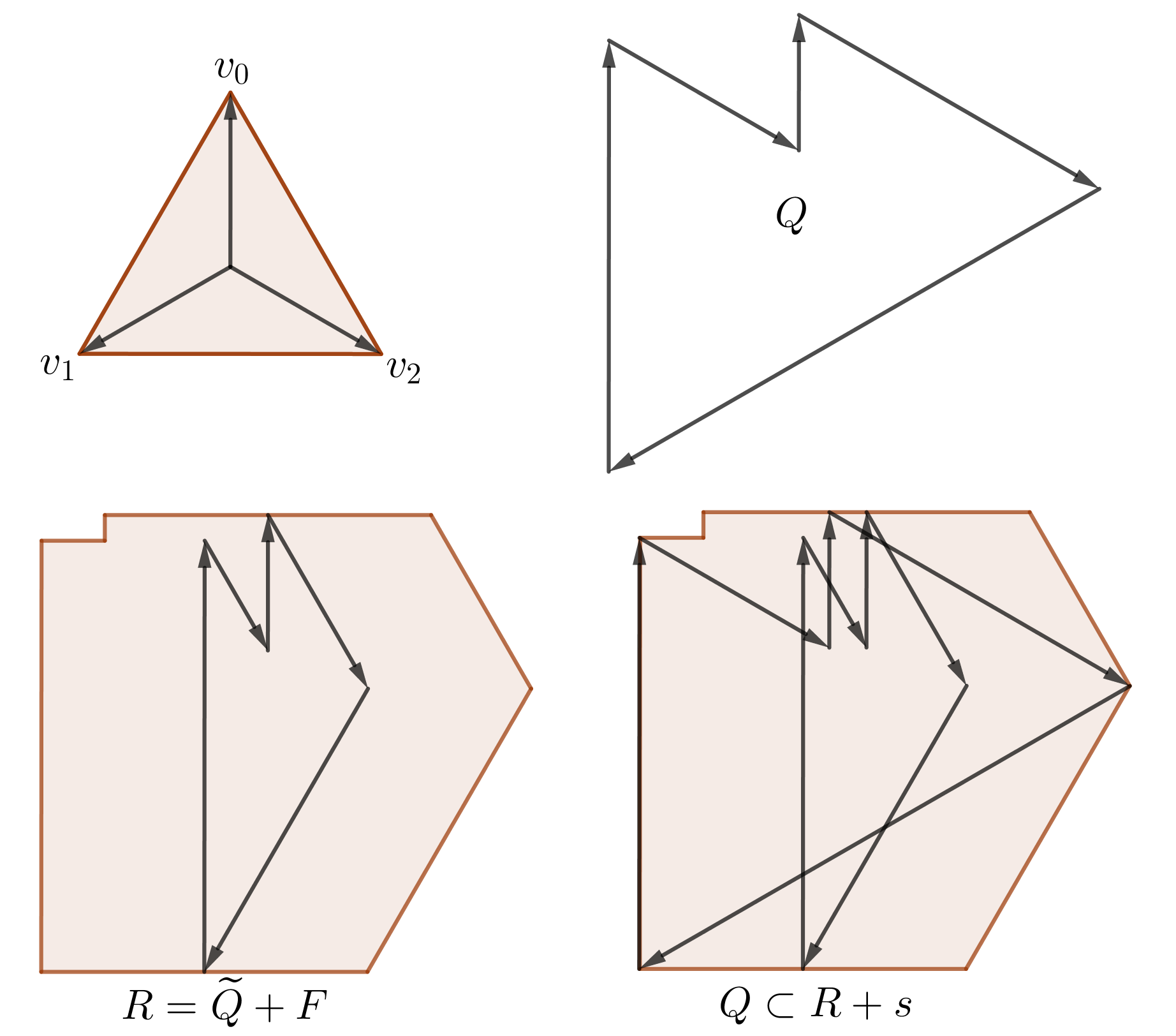}
\caption{$Q$ covered by $R$}
\label{pic:permutfit1}
\end{figure}

The next step is to cover the set $R$ by $P_n$ (see Figure~\ref{pic:permutfit2}). To do this, we consider the coordinates $(\widetilde q_1(t), \ldots, \widetilde q_n(t))$ of the point $\widetilde q(t)$ in the basis $\widetilde v_1, \ldots, \widetilde v_n$. Shift $\widetilde Q$ (and $R = \widetilde Q + F$ accordingly) so that, for any $i$,
$$
\min\limits_{t\in [0,n+1]} \widetilde q_i(t) = 0.
$$
We claim that $R \subset P_n$ after such a shift. Indeed, any point of $R$ can be represented as
$$
f+ \widetilde q(t) = f + \sum\limits_{i=1}^n \widetilde q_i(t)\widetilde v_i,
$$
for some $t\in [0,n+1]$ and $f \in F$. From the assumptions of the lemma, it follows that $\widetilde q_i(t) \in [0,n+1]$. Hence, the point $\sum\limits_{i=1}^n \widetilde q_i(t) \widetilde v_i$ belongs to the Minkowski sum of the non-horizontal segments
$$
\sum\limits_{i=1}^n [0,(n+1)\widetilde v_i] = \sum\limits_{i=1}^n ([v_0, v_i] - v_0).
$$
Hence,
$$
f + \sum\limits_{i=1}^n \widetilde q_i(t) \widetilde v_i \in F + \sum\limits_{i=1}^n ([v_0, v_i] - v_0) = \sum\limits_{1\le i < j\le n} [v_i, v_j] + \sum\limits_{i=1}^n [v_0, v_i] = P_n,
$$
as required.
\end{proof}
\begin{figure}[h]
\centering
\includegraphics[width=0.45\textwidth]{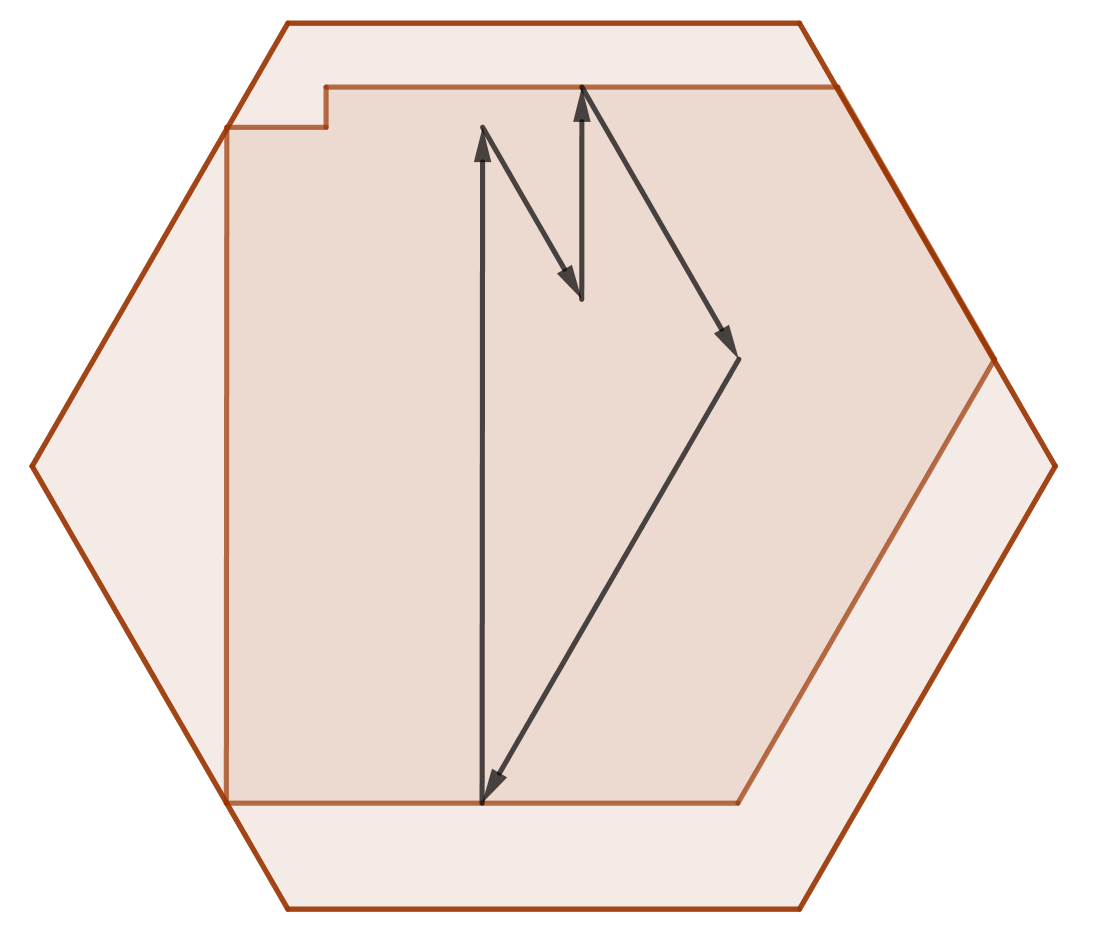}
\caption{$R$ covered by $P_n$}
\label{pic:permutfit2}
\end{figure}

\begin{remark}
\label{rem:uppercapacity}

As a referee pointed out, an example of a two-bouncing trajectory in the configuration $P_n \times \triangle_n^\circ$ estimates from above the cylindrical capacity $c_{cyl}(P_n \times \triangle_n^\circ)$ by $(n+1)^2$ (see~\cite[Remark 4.2]{artstein2014from}). So we have
$$
(n+1)^2 \le c_{HZ}(P_n \times \triangle_n^\circ) \le c_{cyl}(P_n \times \triangle_n^\circ) \le (n+1)^2,
$$

and all the capacities between $c_{HZ}$ and $c_{cyl}$ coincide on the Lagrangian product $P_n \times \triangle_n^\circ$.
Consequently, the equality in (\ref{eq:viterbo}) holds for $X = P_n \times \triangle_n^\circ$ and any symplectic capacity greater than the Hofer--Zehnder capacity.

\end{remark}

\section{Special cases of Viterbo's conjecture}
\label{sec:inequality}

Viterbo's conjecture is proven~\cite{hermann1998non} for ellipsoids, polydiscs and convex Reinhardt domains. However, even for $X \subset \mathbb{R}^4$ Viterbo's conjecture remains widely open. Here we prove the following special cases of the conjecture.

\begin{theorem}
\label{thm:viterbosimplex}
Let $K \subset V = \mathbb{R}^n$ be a convex body. Let $\triangle_n \subset V^*$ be a simplex. Then
$$
\volu (K \times \triangle_n) \ge \frac{c_{HZ}(K \times \triangle_n)^n}{n!},
$$
where $K$ and $\triangle_n$ lie in Lagrangian subspaces.
\end{theorem}

\begin{theorem}
\label{thm:viterbocube}
Let $K \subset V = \mathbb{R}^n$ be a convex body. Let $\square_n \subset V^*$ be a parallelotope. Then
$$
\volu (K \times \square_n) \ge \frac{c_{HZ}(K \times \square_n)^n}{n!},
$$
where $K$ and $\square_n$ lie in Lagrangian subspaces.
\end{theorem}

The latter can be interpreted as a sharp isoperimetric-like inequality for billiard trajectories in the $\ell_1$-norm, with equality attained on the crosspolytope $K = \square_n^\circ$:
$$
\xi_{\square_n}(K) \le 2\left(n! \volu (K)\right)^{1/n}.
$$
Similarly, the former is a sharp isoperimetric-like inequality for billiard trajectories in the non-symmetric $\|\cdot\|_{\triangle_n}$-norm, with equality attained on certain permutohedra.

For the proof of Theorem~\ref{thm:viterbosimplex} we will need a simple topological lemma.

\begin{lemma}
\label{lem:avoidtriangulation}
Let $T$ be a triangulation of $\mathbb{R}^n$ with the diameters of simplices uniformly bounded from above. Let $\Lambda$ be the vertex set of $T$. Assume that every simplex of $T$ can be covered by a translate of a convex body $K \subset \mathbb{R}^n$. Then any translate of $K$ meets $\Lambda$.
\end{lemma}

\begin{proof}
Let $d \in \mathbb{R}$ be an upper bound for the diameters of the cells of $T$.
Suppose there is a translate $K+t$ of $K$ avoiding $\Lambda$. We define a continuous vector field $\nu : \mathbb{R}^n \to \mathbb{R}^n$ as follows. For every vertex $\lambda \in \Lambda$, set $v(\lambda)$ equal to any unit vector pointing away from $K+t$. More precisely, take any hyperplane $H \ni \lambda$ avoiding $K+t$, and set $\nu(\lambda)$ to be the normal vector of $H$ pointing to the halfspace missing $K+t$ (see Figure~\ref{pic:brouwer}, illustrating the two-dimensional case). Then extend $\nu$ affinely to the entire $\mathbb{R}^n$.

\begin{figure}[h]
\centering
\includegraphics[width=0.6\textwidth]{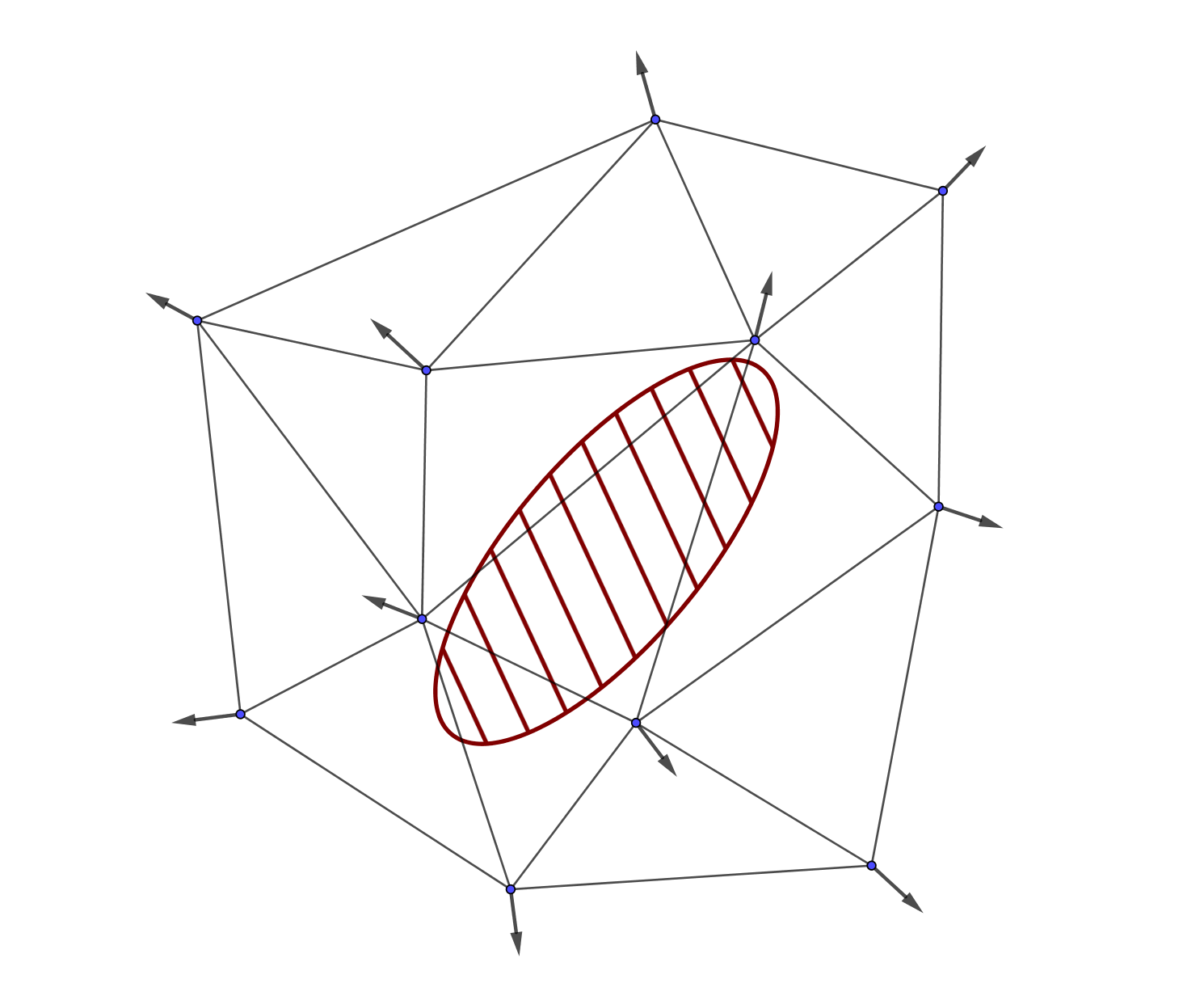}
\caption{Construction of vector field $\nu$}
\label{pic:brouwer}
\end{figure}

Now take a ball $B$ centered at any point of $K+t$. We take it large enough, say, of radius $100 (d + \diam K)$. Then $\frac{\nu}{|\nu|}\vert_{\partial B} : \partial B \to S^{n-1}$ has degree 1, hence there is $x \in \inte B$ such that $v(x) = 0$. If $x$ lies in a simplex $\sigma$ of $T$ with vertices $v_0, \ldots, v_n$, then there are non-negative multipliers $\alpha_0, \ldots, \alpha_n$, not all zeros, such that $\sum\limits_i \alpha_i \nu(v_i) = 0$. Informally, this means that $K+t$ is ``blocked'' inside the ``cage'' $\sigma$ with the ``cage bars'' $v_i$. Let us show rigorously that in this situation $\sigma$ cannot be covered by a translate of $K$.
We know that $K+t$ does not cover the vertices of $\sigma$. The choice of $\nu(v_i)$ implies that $\langle u - v_i, \nu(v_i) \rangle < 0$ for any $u \in K+t$. Assume that there is a translate $K+t+s$ covering $\sigma$. Multiplying $s$ by $\sum\limits_i \alpha_i \nu(v_i)$, we get
$$
\sum\limits_{i=0}^n \alpha_i \langle s, \nu(v_i) \rangle = 0.
$$
It follows that $\langle s, \nu(v_i) \rangle \le 0$ for some $i$. Then we obtain a contradiction as follows:
$$
0 = \langle v_i - v_i, \nu(v_i) \rangle = \langle \underbrace{(v_i-s)}_{\in K+t} - v_i, \nu(v_i) \rangle + \langle s, \nu(v_i) \rangle < 0.
$$
\end{proof}

\begin{proof}[Proof of Theorem~\ref{thm:viterbosimplex}]

The inequality in question is invariant under the following family of transformations: if $V$ is acted on by an affine transform $A$, then $V^*$ is acted on by $(A^*)^{-1}$. Any simplex is affinely equivalent to a regular one, so we can assume that $\triangle_n^\circ = \conv\{u_0, \ldots, u_n\}$ is a regular simplex centered at the origin, with $|u_0| = \ldots = |u_n| = 1$. Further, we scale $K$ so that $c_{HZ}(K \times \triangle_n)$ becomes equal to $(n+1)$.

Consider the lattice $\Lambda$ generated by the vectors $u_0, \ldots, u_n$; recall that $\Lambda$ is just a scaled copy of $A_n^*$, and its Vorono\u{\i} cell $P$ (centered at the origin) is congruent to $\frac{1}{\sqrt{n^2+n}} \widetilde{P_n}$.
We will prove that $\volu K \ge \volu P$. This will suffice:
$$
\volu (K \times \triangle_n) \ge \volu (P \times \triangle_n) = \frac{c_{HZ}(P \times \triangle_n)^n}{n!} = \frac{(n+1)^n}{n!} = \frac{c_{HZ}(K \times \triangle_n)^n}{n!}.
$$
Here we used the results of Section~\ref{sec:equality}. Note that in our scale, $P$ equals $\frac{1}{n+1}$-times the Minkowski sum of the edges of $\triangle_n^\circ$, hence $c_{HZ}(P \times \triangle_n) = n+1$ and $\volu (P \times \triangle_n) = \frac{c_{HZ}(P \times \triangle_n)^n}{n!}$.

Now we are proving the estimate $\volu K \ge \volu P$. We will do it by showing that
$$
\bigcup_{\lambda \in \Lambda} (\lambda + K) = \mathbb{R}^n.
$$
This will imply that $\volu K$ is not less than the volume of the fundamental domain of $\Lambda$, i.e., $\volu P$.

Consider the Delaunay triangulation $T$ of $\mathbb{R}^n$ with respect to $\Lambda$. For each simplex $\sigma$ of $T$, we invoke Lemma~\ref{lem:equiedge} to find the polygonal line $Q_\sigma$ of length $n+1$. Since $c_{HZ}(K \times \triangle_n) = n+1$, such $Q_\sigma$ can be covered by a translate of $K$ (by Theorem~\ref{thm:bezdeks}). Therefore, any simplex of $T$ can be covered by a translate of $K$.

Assume that $\bigcup\limits_{\lambda \in \Lambda} (\lambda + K) \neq \mathbb{R}^n$, that is, there exists $x \in \mathbb{R}^n$ such that $(x + \Lambda) \cap K = \varnothing$. So we found a translate $K - x$ of $K$ that avoids $\Lambda$, while every simplex of $T$ can be covered by a translate of $K$. This contradicts Lemma~\ref{lem:avoidtriangulation}.
\end{proof}

Theorem~\ref{thm:viterbocube} will be proved inductively using the following lemma.

\begin{lemma}
\label{lem:viterboaddsegment}
Let $L \subset \mathbb{R}^{n-1}$ be a convex body. Suppose we established the inequality
$$
\volu (M \times L) \ge \frac{c_{HZ}(M \times L)^{n-1}}{(n-1)!}
$$
for all convex bodies $M \subset \mathbb{R}^{n-1}$. Then for any convex body $K \subset \mathbb{R}^n$ the following holds:
$$
\volu (K \times (L \times [-1,1])) \ge \frac{c_{HZ}(K \times (L \times [-1,1]))^n}{n!}.
$$
(As usual, the relevant products are assumed to be Lagrangian.)
\end{lemma}

\begin{proof}
Without loss of generality assume that $L$ contains the origin in its interior. Denote by $H$ the hyperplane in $\mathbb{R}^n$ where $L$ sits.
Note that
$$
(L \times [-1,1])^\circ = \conv (L^\circ \times \{0\} \cup \{\underbrace{(0,\ldots,0}_{n-1})\} \times [-1,1]).
$$
Sometimes this body is called the \emph{free sum} or the \emph{$\ell_1$-sum} of $L^\circ$ and $[-1,1]$.

Scale $K$ so that $c_{HZ}(K \times L \times [-1,1]) = \xi_{L \times [-1,1]}(K) = 4$. This implies that any closed polygonal line of $\|\cdot\|_{L \times [-1,1]}$-length 4 fits into $K$. In particular, any such polygonal line that is parallel to $H$ fits into $\pi_H(K)$, the orthogonal projection of $K$ onto $H$. Thus, $c_{HZ}(\pi_H(K) \times L) \ge 4$. By the assumption,
$$
\volu (\pi_H(K) \times L) \ge \frac{c_{HZ}(\pi_H(K) \times L)^{n-1}}{(n-1)!} \ge \frac{4^{n-1}}{(n-1)!}.
$$
Now we apply the following inequality (proven by Rogers and Shephard~\cite{rogers1958convex} in greater generality):

$$
\volu_n (K) \ge \frac{1}{n} \volu_{n-1} (\pi_H(K)) \cdot \max\limits_{x \in \mathbb{R}^n} \volu_1 ((x + H^\bot) \cap K).
$$

Since $\xi_{L \times [-1,1]}(K) = 4$, $K$ contains a segment orthogonal to $H$ of $\|\cdot\|_{L \times [-1,1]}$-length 2, and we conclude that $\max\limits_{x \in \mathbb{R}^n} \volu_1 ((x + H^\bot) \cap K) \ge 2$.
Finally, we combine all inequalities above:
\begin{multline*}
\volu (K \times L \times [-1,1]) \ge \frac{2}{n} \volu (\pi_H(K)) \volu(L \times [-1,1]) = \frac{4}{n} \volu (\pi_H(K) \times L) \ge \\
\ge \frac{4}{n} \cdot \frac{4^{n-1}}{(n-1)!} = \frac{c_{HZ}(K \times L \times [-1,1])^n}{n!}.
\end{multline*}

\end{proof}

\begin{proof}[Proof of Theorem~\ref{thm:viterbocube}]

The inequality in question is invariant under the transforms of the form $A \times (A^*)^{-1}$, for $A$ affine, so we assume that $\square_n = [-1,1]^n$ is a hypercube centered at the origin. Now the claim follows from~\ref{lem:viterboaddsegment} by induction. The base case
$$
\volu ([a,b] \times [-1,1]) \ge c_{HZ}([a,b] \times [-1,1])
$$
is trivial.

\end{proof}

\begin{remark}
\label{rem:viterbosimplexcube}

Starting from~\ref{thm:viterbosimplex} as the induction base, we see that the same induction proves the $c_{HZ}$-version of Viterbo's conjecture for a Lagrangian product of any convex body $K \subset \mathbb{R}^{k+m}$ with a simplex in $\mathbb{R}^k$ and a parallelotope in $\mathbb{R}^m$:

$$
\volu (K \times \triangle_k \times \square_m) \ge \frac{c_{HZ}(K \times \triangle_k \times \square_m)^{k+m}}{(k+m)!}.
$$

\end{remark}
%
%\begin{question}
%\label{ques:viterbocrosspolytope}
%Prove a sharp isoperimetric billiard inequality in the $\ell_\infty$-norm, that is,
%$$
%\volu (K \times \lozenge_n) \ge \frac{c_{HZ}(K \times \lozenge_n)^n}{n!}.
%$$
%\end{question}
%
%\begin{question}
%\label{ques:viterboball}
%In a sharp isoperimetric billiard inequality in the Euclidean norm,
%$$
%\xi_{\bigcirc_n}(K) \le C_n \volu K ^{1/n},
%$$
%what is the optimal value of $C_n$, at least, for $n=2$?
%\end{question}

\bibliography{permut}

\begin{thebibliography}{10}

\bibitem{akopyan2016elementary}
A.~Akopyan, A.~Balitskiy, R.~Karasev, and A.~Sharipova.
\newblock Elementary approach to closed billiard trajectories in asymmetric
  normed spaces.
\newblock {\em Proc. Amer. Math. Soc.}, 144(10):4501--4513, 2016.

\bibitem{paiva2014contact}
J.~Alvarez~Paiva and F.~Balacheff.
\newblock Contact geometry and isosystolic inequalities.
\newblock {\em Geom. Funct. Anal.}, 24(2):648--669, 2014.

\bibitem{artstein2014from}
S.~Artstein-Avidan, R.~Karasev, and Y.~Ostrover.
\newblock From symplectic measurements to the {M}ahler conjecture.
\newblock {\em Duke Math. J.}, 163(11):2003--2022, 2014.

\bibitem{artstein2008m}
S.~Artstein-Avidan, V.~Milman, and Y.~Ostrover.
\newblock The m-ellipsoid, symplectic capacities and volume.
\newblock {\em Comment. Math. Helv.}, 83(2):359--369, 2008.

\bibitem{artstein2014bounds}
S.~Artstein-Avidan and Y.~Ostrover.
\newblock Bounds for {M}inkowski billiard trajectories in convex bodies.
\newblock {\em Int. Math. Res. Not.}, 2014(1):165--193, 2014.

\bibitem{balitskiy2016shortest}
A.~Balitskiy.
\newblock Shortest closed billiard trajectories in the plane and equality cases
  in {M}ahler's conjecture.
\newblock {\em Geom. Dedicata}, 184(1):121--134, 2016.

\bibitem{bezdek2009shortest}
D.~Bezdek and K.~Bezdek.
\newblock Shortest billiard trajectories.
\newblock {\em Geom. Dedicata}, 141(1):197--206, 2009.

\bibitem{conway2013sphere}
J.~H. Conway and N.~J.~A. Sloane.
\newblock {\em Sphere packings, lattices and groups}, volume 290.
\newblock Springer Science \& Business Media, 2013.

\bibitem{gruber2007convex}
P.~M. Gruber.
\newblock {\em Convex and discrete geometry}, volume 336.
\newblock Springer Science \& Business Media, 2007.

\bibitem{gutkin2002billiards}
E.~Gutkin and S.~Tabachnikov.
\newblock Billiards in finsler and minkowski geometries.
\newblock {\em J. Geom. Phys.}, 40(3):277--301, 2002.

\bibitem{hermann1998non}
D.~Hermann.
\newblock Non-equivalence of symplectic capacities for open sets with
  restricted contact type boundary.
\newblock 1998.

\bibitem{latschev2013gromov}
J.~Latschev, D.~McDuff, and F.~Schlenk.
\newblock The gromov width of 4--dimensional tori.
\newblock {\em Geom. Topol.}, 17(5):2813--2853, 2013.

\bibitem{postnikov2009permutohedra}
A.~Postnikov.
\newblock Permutohedra, associahedra, and beyond.
\newblock {\em Int. Math. Res. Not.}, 2009(6):1026--1106, 2009.

\bibitem{rogers1958convex}
C.~A. Rogers and G.~C. Shephard.
\newblock Convex bodies associated with a given convex body.
\newblock {\em J. Lond. Math. Soc.}, 1(3):270--281, 1958.

\bibitem{schlenk2005embedding}
F.~Schlenk.
\newblock {\em Embedding problems in symplectic geometry}, volume~40.
\newblock Walter de Gruyter, 2005.

\bibitem{viterbo2000metric}
C.~Viterbo.
\newblock Metric and isoperimetric problems in symplectic geometry.
\newblock {\em J. Amer. Math. Soc.}, 13(2):411--431, 2000.

\bibitem{ziegler1995lectures}
G.~M. Ziegler.
\newblock {\em Lectures on polytopes}, volume 152.
\newblock Springer Science \& Business Media, 1995.

\end{thebibliography}
\bibliographystyle{abbrv}
\end{document}